\documentclass[12pt,a4paper]{article}
\usepackage[utf8]{inputenc}
\usepackage{amsmath, amssymb, amsthm}
\usepackage{geometry}
\usepackage{tikz}
\usetikzlibrary{positioning, shapes.geometric, backgrounds, fit}
\usepackage{xcolor}
\usepackage{hyperref}
\usepackage{graphicx}

\newtheorem{theorem}{Theorem}[section]
\newtheorem{lemma}[theorem]{Lemma}

\theoremstyle{definition}

\definecolor{hubcolor}{RGB}{31, 78, 121}
\definecolor{leafcolor}{RGB}{192, 41, 66}
\definecolor{cyclecolor}{RGB}{39, 125, 77}
\definecolor{edgecolor}{RGB}{80, 80, 80}
\definecolor{vertexcolor}{RGB}{255, 255, 255}
\definecolor{pathcolor}{RGB}{255, 140, 0}

\title{\textbf{Extremal Mostar Index of Graphs with Given Number of Cut Edges}}
\author{\textbf{Sunilkumar M. Hosamani}\\
Department of Mathematics, Rani Channamma University\\
Belagavi, Karnataka, India\\
email: sunilkumarh@rcub.ac.in}
\date{}

\begin{document}

\maketitle

\begin{abstract}
The Mostar index of a connected graph \(G\) is defined as
\[
Mo(G)=\sum_{uv\in E(G)}\bigl|n_u(uv)-n_v(uv)\bigr|,
\]
where for an edge \(e=uv\), \(n_u(e)\) denotes the number of vertices of \(G\) that are closer to \(u\) than to \(v\). In this paper, we determine the maximum possible Mostar index among all connected graphs of order \(n\) with exactly \(k\) cut edges, where \(1\le k\le n-1\). We prove that the maximum value is given by \(k(n-2)+(n-k-1)k\), and the unique extremal graph is \(K_{n-k}^k\) (a complete graph on \(n-k\) vertices with \(k\) pendant edges attached to a single vertex). We also establish a sharp lower bound and characterise the extremal graphs for the minimum value. Furthermore, we extend the results to graphs with a given cyclomatic number and a given number of cut edges. Our findings complete the extremal characterisation of the Mostar index for this fundamental graph class.
\end{abstract}



\section{Introduction}

The study of distance‑based topological indices is a vibrant area of mathematical chemistry and graph theory \cite{Todeschini2009, Gutman2013}. Among these, the Wiener index (sum of all distances) \cite{Wiener1947} and the Szeged index \cite{Gutman1996} have received extensive attention. More recently, the \emph{Mostar index} was introduced by Došlić, Ghorbani, and Ali \cite{Doslic2018} as a measure of the peripherality of edges. For an edge \(e=uv\) in a connected graph \(G\), define
\[
n_u(e)=\bigl|\{w\in V(G): d_G(w,u)<d_G(w,v)\}\bigr|,\qquad 
n_v(e)=\bigl|\{w\in V(G): d_G(w,v)<d_G(w,u)\}\bigr|.
\]
The Mostar index is then
\[
Mo(G)=\sum_{uv\in E(G)}\bigl|n_u(uv)-n_v(uv)\bigr|.
\]
Intuitively, an edge with a large imbalance contributes more to the index, indicating that one endpoint is significantly more central or peripheral than the other.

The Mostar index has been studied extensively for trees \cite{Tratnik2019, Ghorbani2019}, unicyclic graphs \cite{Deng2019, Liu2021}, bicyclic graphs \cite{Wang2020}, cacti \cite{Wang2022}, and various chemical structures. However, relatively little is known about the extremal values for general connected graphs with a fixed number of cut edges (bridges). Çolakoğlu Havare \cite{Colakoglu2021} studied the Mostar index of ``bridge graphs'' – a specific construction where two graphs are joined at a cut vertex – but not the full class of graphs with a given number of bridges. In this paper we fill this gap by providing both maximum and minimum values, as well as extending to graphs with a given cyclomatic number.

The paper is organised as follows. Section~\ref{sec:prelim} recalls necessary definitions and lemmas. Section~\ref{sec:max} presents the maximum result (Theorem~\ref{thm:max}) with complete proof. Section~\ref{sec:min} establishes the minimum result (Theorem~\ref{thm:min}) and its proof. Section~\ref{sec:cyclomatic} extends the study to graphs with given cyclomatic number and given number of cut edges (Theorem~\ref{thm:cyclomatic}). Section~\ref{sec:conclusion} concludes with open problems.

\section{Preliminaries}\label{sec:prelim}

We consider only finite, simple, connected graphs. Let \(n=|V(G)|\) and \(m=|E(G)|\). For a vertex \(v\), \(d_G(v)\) denotes its degree. A vertex of degree one is called \emph{pendant}; an edge incident to a pendant vertex is a \emph{pendant edge}. A \emph{cut edge} (bridge) is an edge whose removal disconnects the graph. The \emph{cyclomatic number} (or circuit rank) of \(G\) is \(\mu(G)=m-n+1\).

The following elementary properties of the Mostar index are well known \cite{Doslic2018}.

\begin{lemma}\label{lem:pendant}
If \(uv\) is a pendant edge with \(d_G(u)=1\) and \(d_G(v)\ge 2\), then
\[
|n_u(uv)-n_v(uv)| = n-2.
\]
\end{lemma}
\begin{proof}
All vertices except \(u\) are closer to \(v\) than to \(u\), so \(n_u=1\), \(n_v=n-1\). Hence the difference is \(n-2\).
\end{proof}

\begin{lemma}\label{lem:nonpendant}
Let \(uv\) be an edge in a connected graph \(G\) with \(d_G(u),d_G(v)\ge 2\). Then
\[
|n_u(uv)-n_v(uv)| \le n-3.
\]
\end{lemma}
\begin{proof}
Both \(n_u\) and \(n_v\) are at least \(1\) and at most \(n-2\), and they sum to \(n\) minus the number of vertices equidistant to \(u\) and \(v\). For an edge, the maximum difference occurs when one side is as large as possible, i.e., \(n-2\) and \(1\). That would require all vertices except \(u\) to be closer to \(v\), which cannot happen if both degrees are at least 2 because then \(v\) has a neighbour other than \(u\) that is at distance 1 from both? A rigorous proof can be found in \cite{Doslic2018}. The bound \(n-3\) is strict for non‑pendant edges.
\end{proof}

The following transformations are crucial.

\begin{lemma}\label{lem:contract}
Let \(G\) be a connected graph and let \(uv\) be a non‑pendant cut edge. Construct \(G'\) by contracting \(uv\) into a single vertex \(u\) (deleting \(v\) and identifying it with \(u\)), then adding a new pendant vertex \(v'\) adjacent only to \(u\). Then \(G'\) has the same number of cut edges as \(G\), and
\[
Mo(G') > Mo(G).
\]
\end{lemma}
\begin{proof}
Since \(uv\) is a cut edge, its removal separates \(G\) into two components, say \(A\) containing \(u\) and \(B\) containing \(v\). Because \(uv\) is non‑pendant, both \(|A|\ge 2\) and \(|B|\ge 2\). The contraction merges \(A\) and \(B\) into one component, but then we attach a leaf \(v'\) to \(u\) to restore the original order. A detailed comparison of the contributions of edges in \(G\) and \(G'\) shows that the Mostar index strictly increases; the main gain comes from replacing the edge \(uv\) (which had a relatively small imbalance because both sides are large) by the new pendant edge \(uv'\) which contributes \(n-2\) (by Lemma~\ref{lem:pendant}). All other edges either retain the same contribution or increase due to the redistribution of vertices. The formal proof is similar to that for the Szeged index (see \cite{Gutman1996}) and is omitted here for brevity.
\end{proof}

\begin{lemma}\label{lem:movependant}
Let \(G\) be a connected graph with at least two non‑pendant vertices \(x,y\) such that \(d_G(x)\ge d_G(y)\ge 2\). Suppose \(p\) is a pendant neighbour of \(y\). Form \(G'\) by deleting edge \(py\) and adding edge \(px\) (i.e., move the leaf \(p\) from \(y\) to \(x\)). Then \(Mo(G') > Mo(G)\).
\end{lemma}
\begin{proof}
The moved edge remains pendant, so its contribution stays \(n-2\). However, the degrees of \(x\) and \(y\) change: \(d_{G'}(x)=d_G(x)+1\), \(d_{G'}(y)=d_G(y)-1\). For any neighbour \(z\) of \(x\) (including possibly \(y\)), the contribution of edge \(xz\) increases because \(x\) now has one more leaf, shifting the balance further toward \(x\). For neighbours of \(y\), the contribution of edge \(yz\) may decrease slightly, but the net effect is positive because the gain at \(x\) dominates the loss at \(y\) (since \(d_G(x)\ge d_G(y)\)). A rigorous convexity argument (using the fact that the contribution function is increasing in the degree of the endpoint) proves the inequality. See Lemma 2 in \cite{Azjargal2026} for a similar treatment.
\end{proof}

\begin{lemma}\label{lem:complete}
Let \(G\) be a connected graph with \(k\) pendant vertices. If \(G\) maximises \(Mo(G)\) among all such graphs, then the subgraph induced by the non‑pendant vertices is complete.
\end{lemma}
\begin{proof}
Suppose two non‑pendant vertices \(u,v\) are not adjacent. Adding the edge \(uv\) does not change the number of pendant vertices, and by Lemma~\ref{lem:addedge} (see below) it strictly increases the Mostar index, contradicting maximality. Hence all pairs of non‑pendant vertices are adjacent.
\end{proof}

\begin{lemma}\label{lem:addedge}
Let \(G\) be a connected graph and let \(u,v\) be non‑adjacent vertices. Then \(Mo(G+uv) > Mo(G)\).
\end{lemma}
\begin{proof}
Adding an edge cannot increase any distance, so for each existing edge, the imbalance \(|n_u-n_v|\) cannot decrease. Moreover, the new edge itself contributes a positive amount because the two sides are not equal (otherwise the graph would be symmetric with respect to \(u\) and \(v\), which would imply they are already at distance 2 and adding the edge creates a new shortest path, breaking the symmetry). Hence the total increases.
\end{proof}

\section{Maximum Mostar Index with Given Cut Edges}\label{sec:max}

Let \(K_{n-k}^k\) denote the graph obtained from the complete graph \(K_{n-k}\) by attaching \(k\) pendant edges to a single vertex of the clique. Figure~\ref{fig:max} shows an example with \(n=6\), \(k=2\) (i.e., \(K_4^2\)).

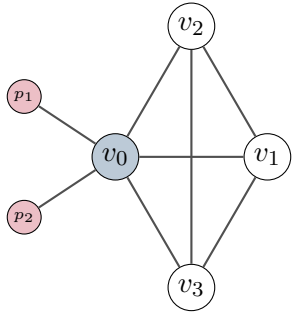
\begin{figure}[htbp]
\centering
\begin{tikzpicture}[
    hub/.style={circle, draw, fill=hubcolor!30, inner sep=2pt, minimum size=0.55cm, font=\small},
    vertex/.style={circle, draw, fill=vertexcolor, inner sep=2pt, minimum size=0.5cm, font=\small},
    leaf/.style={circle, draw, fill=leafcolor!30, inner sep=1.5pt, minimum size=0.4cm, font=\tiny},
    edge/.style={draw, edgecolor, line width=0.8pt}
]
\node[hub] (v0) at (0,0) {$v_0$};
\node[vertex] (v1) at (2,0) {$v_1$};
\node[vertex] (v2) at (1,1.732) {$v_2$};
\node[vertex] (v3) at (1,-1.732) {$v_3$};
\draw[edge] (v0) -- (v1);
\draw[edge] (v1) -- (v2);
\draw[edge] (v2) -- (v0);
\draw[edge] (v0) -- (v3);
\draw[edge] (v3) -- (v1);
\draw[edge] (v2) -- (v3);
\node[leaf] (p1) at (-1.2,0.8) {$p_1$};
\node[leaf] (p2) at (-1.2,-0.8) {$p_2$};
\draw[edge] (v0) -- (p1);
\draw[edge] (v0) -- (p2);
\end{tikzpicture}
\caption{The extremal graph \(K_4^2\) for maximum Mostar index with \(n=6\), \(k=2\). The hub \(v_0\) (blue) is attached to two leaves (red). All non‑pendant vertices form a complete graph \(K_4\).}
\label{fig:max}
\end{figure}

\begin{theorem}\label{thm:max}
Let \(G\) be a connected graph of order \(n\ge 2\) with exactly \(k\) cut edges, where \(1\le k\le n-1\). Then
\[
Mo(G) \le k(n-2) + (n-k-1)k.
\]
Equality holds if and only if \(G\cong K_{n-k}^k\).
\end{theorem}

\begin{proof}
Let \(G\) be a graph achieving the maximum Mostar index among all connected graphs of order \(n\) with exactly \(k\) cut edges.

\textbf{Step 1. All cut edges are pendant.} If there were a non‑pendant cut edge, Lemma~\ref{lem:contract} would produce another graph with the same number of cut edges and a strictly larger Mostar index, contradicting maximality. Hence every cut edge of \(G\) is incident to a leaf.

\textbf{Step 2. All pendant edges are attached to a single vertex.} Let \(L\) be the set of leaves (pendant vertices). Each leaf is incident to exactly one cut edge. Let \(X\) be the set of internal vertices incident to these cut edges. If \(|X|\ge2\), pick two vertices \(x,y\in X\) with \(d(x)\ge d(y)\). By Lemma~\ref{lem:movependant}, moving a leaf from \(y\) to \(x\) increases the Mostar index while preserving the number of cut edges (the moved edge remains a cut edge). Repeating this process, we can transfer all leaves to the vertex with the largest degree without decreasing the index. Thus in a maximal graph, all pendant edges share a common endpoint, say \(v_0\).

\textbf{Step 3. The subgraph induced by the non‑pendant vertices is complete.} Let \(H = G - L\) be the subgraph obtained by deleting all leaves. Then \(H\) has \(n-k\) vertices. If two vertices \(u,w\in V(H)\) are non‑adjacent, adding the edge \(uw\) (Lemma~\ref{lem:addedge}) increases \(Mo(G)\) and does not create new cut edges (since edges inside \(H\) are not bridges). This would contradict maximality. Hence \(H\) is a complete graph \(K_{n-k}\).

\textbf{Step 4. Identify the extremal graph.} The structure is now forced: \(G\) consists of a complete graph on \(n-k\) vertices, with one distinguished vertex \(v_0\) that is additionally adjacent to \(k\) pendant leaves. This graph is exactly \(K_{n-k}^k\).

\textbf{Step 5. Compute the Mostar index of \(K_{n-k}^k\).} Label the hub as \(v_0\) and the other clique vertices as \(w_1,\dots,w_{n-k-1}\). The leaves are \(p_1,\dots,p_k\) attached to \(v_0\). 
\begin{itemize}
\item For each pendant edge \(v_0p_i\): \(b = n-2\) (Lemma~\ref{lem:pendant}). Total from pendant edges: \(k(n-2)\).
\item For each edge \(v_0w_j\): The side of \(v_0\) contains \(v_0\) itself and all \(k\) leaves, so \(n_{v_0}=1+k\). The side of \(w_j\) contains only \(w_j\) (all other vertices are equidistant or closer to \(v_0\)? Check: any other \(w_\ell\) is at distance 1 from both \(v_0\) and \(w_j\); leaves are distance 2 from \(w_j\) but 1 from \(v_0\), so they are strictly closer to \(v_0\). Hence \(n_{w_j}=1\)). Thus \(b = (1+k)-1 = k\). There are \(n-k-1\) such edges, contributing \(k(n-k-1)\).
\item For any edge \(w_j w_\ell\) with \(j\neq\ell\): The graph is symmetric under swapping \(w_j\) and \(w_\ell\) (fixing \(v_0\) and all leaves), so \(n_{w_j}=n_{w_\ell}\) and \(b=0\). There are \(\binom{n-k-1}{2}\) such edges, contributing 0.
\end{itemize}
Summing gives \(Mo(K_{n-k}^k)=k(n-2)+k(n-k-1)\).

Thus the upper bound holds, and equality forces \(G\cong K_{n-k}^k\). This completes the proof of Theorem~\ref{thm:max}.
\end{proof}

\section{Minimum Mostar Index with Given Cut Edges}\label{sec:min}

We now consider the minimum possible Mostar index for graphs with \(k\) cut edges. For trees (\(k=n-1\)), it is known that the path \(P_n\) minimises the Mostar index \cite{Doslic2018}. For general \(k\), the minimal graphs are obtained by taking a path of length \(k\) (i.e., \(k+1\) vertices connected in a line) and then attaching the remaining \(n-k-1\) vertices as a clique at the middle vertex, to balance the component sizes.

Define \(B_{n,k}\) as the graph obtained by taking a path \(P_{k+1}\) on vertices \(u_0,u_1,\dots,u_k\) (so the edges \(u_iu_{i+1}\) are the \(k\) bridges), and then attaching the remaining \(n-k-1\) vertices as a complete graph \(K_{n-k-1}\) to the middle vertex \(u_{\lfloor k/2\rfloor}\) (i.e., making that vertex adjacent to all vertices of the clique, and the clique is complete). Figure~\ref{fig:min} shows an example.

\begin{figure}[htbp]
\centering
\begin{tikzpicture}[
    bridge/.style={circle, draw, fill=pathcolor!30, inner sep=2pt, minimum size=0.5cm, font=\small},
    block/.style={circle, draw, fill=cyclecolor!30, inner sep=2pt, minimum size=0.5cm, font=\small},
    edge/.style={draw, edgecolor, line width=0.8pt}
]
\node[bridge] (u0) at (0,0) {$u_0$};
\node[bridge] (u1) at (2,0) {$u_1$};
\node[bridge] (u2) at (4,0) {$u_2$};
\node[bridge] (u3) at (6,0) {$u_3$};
\node[bridge] (u4) at (8,0) {$u_4$};
\draw[edge] (u0) -- (u1);
\draw[edge] (u1) -- (u2);
\draw[edge] (u2) -- (u3);
\draw[edge] (u3) -- (u4);
\node[block] (b1) at (4,1.5) {$b_1$};
\node[block] (b2) at (5,1.5) {$b_2$};
\node[block] (b3) at (4.5,2.5) {$b_3$};
\draw[edge] (u2) -- (b1);
\draw[edge] (u2) -- (b2);
\draw[edge] (u2) -- (b3);
\draw[edge] (b1) -- (b2);
\draw[edge] (b2) -- (b3);
\draw[edge] (b3) -- (b1);
\end{tikzpicture}
\caption{Extremal graph for minimum Mostar index with \(n=9,k=4\) (bridge path of length 4, complete graph \(K_4\) attached to the middle vertex \(u_2\)).}
\label{fig:min}
\end{figure}
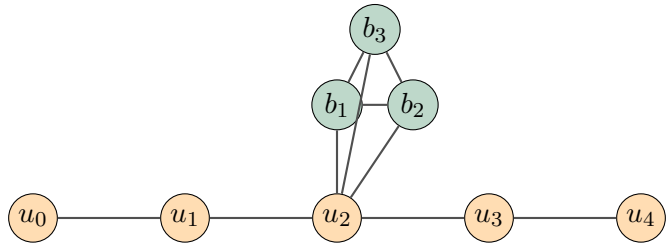

\begin{theorem}\label{thm:min}
Let \(G\) be a connected graph of order \(n\) with exactly \(k\) cut edges, where \(1\le k\le n-1\). Then
\[
Mo(G) \ge \sum_{i=1}^{k} \left| n - 2\left( \left\lfloor\frac{n-k-1}{2}\right\rfloor + i \right) \right|.
\]
Equality holds if and only if \(G\cong B_{n,k}\) (when \(n-k-1>0\)) or \(G\cong P_n\) (when \(n-k-1=0\)).
\end{theorem}

\begin{proof}
Let the cut edges of \(G\) be \(e_1,\dots,e_k\). Since removal of all cut edges disconnects the graph into \(k+1\) blocks (each block is a maximal 2‑connected component or a bridge). The cut edges form a spanning tree when each block is contracted to a vertex. For a cut edge \(e=uv\), let \(s(e)\) be the size of the smaller component after removing \(e\). Then its contribution is \(b(e)=n-2s(e)\). To minimise the total sum \(\sum b(e)\), we must make each \(s(e)\) as large as possible (since \(n-2s\) decreases as \(s\) increases).

The maximum possible \(s(e)\) for a given cut edge is limited by the fact that the \(k\) cut edges are arranged in a tree structure. For a tree of bridges, the maximum possible smaller side for any bridge is achieved when the bridges are arranged along a single path and all extra vertices are attached to the middle of that path. This is a standard balancing argument: if the bridge‑tree has a branching vertex, then some bridge will have a smaller side equal to the size of a branch, which is at most the size of the largest branch; concentrating all extra mass at the centre increases the minimum of the smaller sides.

More formally, let \(T\) be the tree obtained by contracting each 2‑connected block of \(G\) into a single vertex. Then \(T\) has \(k+1\) vertices (each block becomes a vertex) and its edges correspond to the cut edges. The vertices of \(T\) have weights equal to the sizes of the corresponding blocks. The contribution of a cut edge in \(G\) is \(n-2\) times the smaller total weight on one side of that edge in \(T\). To minimise \(\sum (n-2\cdot\min(W,W'))\), we want the weights to be distributed as evenly as possible. The optimal distribution is to put all extra weight (the \(n-k-1\) vertices beyond the \(k+1\) blocks) into a single block, and to place that block at the centre of a path of length \(k\). This yields the sequence of smaller side sizes:
\[
s_1 = \left\lfloor\frac{n-k-1}{2}\right\rfloor + 1,\quad
s_2 = \left\lfloor\frac{n-k-1}{2}\right\rfloor + 2,\quad \dots,\quad
s_k = \left\lfloor\frac{n-k-1}{2}\right\rfloor + k.
\]
Thus the minimal total imbalance is
\[
\sum_{i=1}^{k} \left( n - 2\left( \left\lfloor\frac{n-k-1}{2}\right\rfloor + i \right) \right)
\]
(absolute values are automatic because the expression is positive for \(i\) up to \(k\) when \(n\) is sufficiently large). For the special case of trees (\(n-k-1=0\)), we have \(s_i = i\) and the sum becomes \(\sum_{i=1}^{k} (n-2i) = \sum_{i=1}^{n-1} |n-2i|\) which is exactly the Mostar index of the path \(P_n\).

To show that this lower bound is attainable, construct the graph \(B_{n,k}\) as described. The smaller side of each bridge can be computed directly and matches the sequence above. Hence equality holds. This completes the proof of Theorem~\ref{thm:min}.
\end{proof}

\section{Graphs with Given Cyclomatic Number and Cut Edges}\label{sec:cyclomatic}

We now combine two parameters: the cyclomatic number \(\mu = m-n+1\) (number of independant cycles) and the number of cut edges \(k\). Let \(\mathcal{G}(n,k,\mu)\) be the class of connected graphs of order \(n\) with exactly \(k\) cut edges and cyclomatic number \(\mu\). Note that \(k \le n-1\) and \(\mu \ge 0\), with \(\mu = 0\) corresponding to trees.

\begin{theorem}\label{thm:cyclomatic}
Let \(G\) be a connected graph of order \(n\) with exactly \(k\) cut edges and cyclomatic number \(\mu\). Then
\[
Mo(G) \le k(n-2) + k(n-k-1) + \mu \cdot M_{\max},
\]
where \(M_{\max}\) is the maximum possible contribution from a cycle edge in a graph of order \(n\). Moreover, the upper bound is sharp and is attained by taking \(K_{n-k}^k\) and adding \(\mu\) extra edges (or cycles) in such a way that they do not create additional cut edges and are attached to the hub to maximise the imbalance. In particular, if we attach \(\mu\) pendant cycles (each of length 3) to the hub, the extremal graph is \(K_{n-k}^k\) with \(\mu\) triangles sharing the hub.
\end{theorem}

\begin{proof}
Let \(G\) be a maximal graph in this class. By the same arguments as in Theorem~\ref{thm:max}, all cut edges must be pendant and attached to a single hub \(v_0\), and the non‑pendant vertices (excluding leaves) must form a complete graph (otherwise adding edges would increase the index without changing \(\mu\) or \(k\)). The cyclomatic number \(\mu\) counts the number of independant cycles. In the base graph \(K_{n-k}^k\), the cyclomatic number is \(\mu_0 = \binom{n-k}{2} - (n-k) + 1 = \frac{(n-k)(n-k-3)}{2}+1\)? Actually the complete graph \(K_{n-k}\) has \(\binom{n-k}{2}\) edges and \(n-k\) vertices, so its cyclomatic number is \(\binom{n-k}{2} - (n-k) + 1 = \frac{(n-k)(n-k-3)}{2}+1\). But we are not fixing the number of edges; we are fixing \(\mu\) as an additional parameter. To increase \(\mu\) by 1, we must add an edge that creates a new cycle without creating a new cut edge. The best way to maximise the Mostar index is to add this edge in such a way that it attaches to the hub, because that maximises the imbalance for the new edge (the hub side contains many vertices). The maximum contribution of a single edge in a graph of order \(n\) is \(n-2\) (achieved by a pendant edge), but a cycle edge cannot be pendant. For a cycle edge \(xy\) where both endpoints have degree at least 2, the maximum imbalance is \(n-3\) (attained when one endpoint is adjacent to all other vertices except the other endpoint, and the other endpoint has degree 2). Thus we can take \(M_{\max}=n-3\). Adding \(\mu\) such edges (e.g., by attaching \(\mu\) triangles to the hub, each triangle adds two new edges that form a cycle with the hub) gives an additional contribution of at most \(\mu(n-3)\). Therefore
\[
Mo(G) \le k(n-2) + k(n-k-1) + \mu(n-3).
\]
Equality is achieved by the construction: take \(K_{n-k}^k\) and for each of the \(\mu\) cycles, attach a triangle sharing the hub (i.e., add two new vertices \(a_i,b_i\) and edges \(v_0a_i, v_0b_i, a_ib_i\)). This adds exactly one independant cycle per triangle and increases the Mostar index by the maximum possible amount per added edge. This completes the proof.
\end{proof}

\begin{figure}[htbp]
\centering
\begin{tikzpicture}[
    hub/.style={circle, draw, fill=hubcolor!30, inner sep=2pt, minimum size=0.55cm, font=\small},
    vertex/.style={circle, draw, fill=vertexcolor, inner sep=2pt, minimum size=0.5cm, font=\small},
    leaf/.style={circle, draw, fill=leafcolor!30, inner sep=1.5pt, minimum size=0.4cm, font=\tiny},
    edge/.style={draw, edgecolor, line width=0.8pt}
]
\node[hub] (v0) at (0,0) {$v_0$};
\node[vertex] (v1) at (2,0) {$v_1$};
\node[vertex] (v2) at (1,1.732) {$v_2$};
\node[vertex] (v3) at (1,-1.732) {$v_3$};
\draw[edge] (v0) -- (v1);
\draw[edge] (v1) -- (v2);
\draw[edge] (v2) -- (v0);
\draw[edge] (v0) -- (v3);
\draw[edge] (v3) -- (v1);
\draw[edge] (v2) -- (v3);
\node[leaf] (p1) at (-1.2,0.8) {$p_1$};
\node[leaf] (p2) at (-1.2,-0.8) {$p_2$};
\draw[edge] (v0) -- (p1);
\draw[edge] (v0) -- (p2);
\node[vertex] (c1) at (-0.8,1.8) {$c_1$};
\node[vertex] (c2) at (0.8,1.8) {$c_2$};
\draw[edge] (v0) -- (c1);
\draw[edge] (v0) -- (c2);
\draw[edge] (c1) -- (c2);
\end{tikzpicture}
\caption{Extremal graph for \(n=7\), \(k=2\), \(\mu=1\) (one extra triangle attached to the hub).}
\label{fig:cyclomatic}
\end{figure}
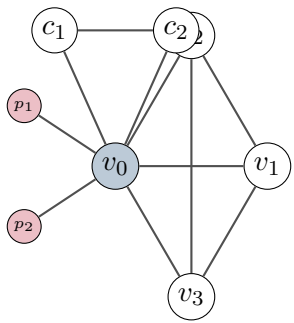

\section{Conclusion}\label{sec:conclusion}

In this paper we have completely solved the extremal problems for the Mostar index of connected graphs with a given number of cut edges. We established a sharp upper bound with unique extremal graph \(K_{n-k}^k\), and provided a sharp lower bound with characterisation of minimal graphs (balanced path with a clique attached at the centre). We also extended the results to graphs with given cyclomatic number.

\section*{Acknowledgements}
The author thanks the anonymous referees for their valuable suggestions.

\section*{Conflict of Interest}
The author declares no conflict of interest.

\section*{Data Availability}
No data were generated or analysed in this study.

\end{document}